\documentclass[12pt]{amsart}

\usepackage{amsfonts, amssymb, amscd}
\usepackage{verbatim}
\usepackage{eucal}
\usepackage{amssymb}
\usepackage{mathrsfs}
\usepackage{graphicx}

\def\ZZ{{\mathbb Z}}

\def\PP{{\mathbb P}}
\def\CC{{\mathbb C}}

\def\Span{{\rm Span}}
\def\Ann{{\rm Ann}}
\def\Im{{\rm Im}}

\addtocounter{footnote}{1}

\newtheorem{lemma}{Lemma}[section]
\newtheorem{theorem}[lemma]{Theorem}

\newtheorem{proposition}[lemma]{Proposition}
\theoremstyle{definition}

\newtheorem{remark}[lemma]{Remark}

\theoremstyle{remark}
\newtheorem*{proof*}{Proof}
\numberwithin{equation}{section}

\title[Threefolds of order one in the six-quadric]{Threefolds of order one 
in the six-quadric}

\author{Lev Borisov and Jeff Viaclovsky}
\address{Department of Mathematics \\ University of Wisconsin \\
  Madison \\ WI \\ 53706 \\ USA}
\email{borisov@math.wisc.edu, jeffv@math.wisc.edu}
\thanks{Lev Borisov has been partially supported
by the National Science Foundation under grant DMS-0758480.
Jeff Viaclovsky has been partially supported
by the National Science Foundation under grant DMS-0804042.}
\date{August 12, 2008}
\begin{document}

\begin{abstract}
Consider the smooth quadric $Q_6$ in $\PP^7$. The
middle homology group $H_6(Q_6,\ZZ)$ is isomorphic to $\ZZ \oplus \ZZ$,
with a basis given by two classes of linear subspaces.
We classify all threefolds of bidegree $(1,p)$ inside $Q_6$.
\end{abstract}

\maketitle

\section{Introduction}

The study of surfaces inside the Grassmannian $G(2,4)$ of lines in
the complex projective space of dimension three has a long and rich history,
we refer the reader to  \cite{Gross} for references. 
The homology class of any such surface is given by two nonnegative integers
$(m,n)$ called the {\em{bidegree}}, with $m$ classically called the {\em{order}} and 
$n$ called the {\em{class}}\hspace{.2mm}. 
In particular, surfaces of order one have been completely classified, 
see \cite{Ran} for a modern treatment.

In this paper we are interested in \emph{threefolds} inside the
six-dimensional nonsingular quadric $Q_6$ in the 
complex projective space $\PP^7$. There are several reasons one 
may want to study these threefolds. First, the 
Grassmannian $G(2,4)$ 
is a nonsingular quadric in $\PP^5$. The nonsingular quadric $Q_6$
is the natural example in the next dimension. Namely, its middle homology
$H_6(Q_6,\ZZ)$ is isomorphic to $\ZZ \oplus \ZZ$ and it is natural to try to study 
threefolds inside $Q_6$ with a given homology class. 
We note that it is fairly easy to prove
that the only threefold of order zero is a linear $\PP^3$ of
bidegree $(0,1)$, so order one is the first interesting case. 
Second, the quadric
$Q_6$ is the orthogonal Grassmannian of isotropic subspaces of dimension four
in $\CC^8$, and thus has special geometric significance.  
Last, but not least, threefolds of bidegree $(1,p)$
in $Q_6$ appear naturally in the study of orthogonal complex structures
on subdomains of the sphere $S^6$, see \cite{BSV}. This was the original motivation
behind our interest in the topic. 

The main result of this paper is that every irreducible $Z$ of bidegree 
$(1,p)$ in $Q_6$ is given by one of the following:
\begin{itemize}
\item
$p=0$ and $Z$ is a horizontal $\PP^3$.
\item
$p=1$ and $Z$ is a smooth quadric in $\PP^4 \subset \PP^7$. 
\item
$p=3$ and $Z$ is a cone over the Veronese surface 
in $\PP^5 \subset \PP^7$. 
\item
$p\geq 1$ and $Z$ is a Weil divisor of bidegree $(1,p)$ in a rank 
four dimension four quadric $Q_4\subset Q_6$.
\end{itemize}
This classification will also enable us to classify the 
{\em{smooth}} threefolds of order one; the only examples are
the first two cases above, and the Segre embedding 
for $p=2$, see Theorem \ref{mainsmooth}. 
The plan of the paper is as follows. In Section \ref{statements}, we
state our main results in more detail,
and describe the last two cases above 
more explicitly in terms of homogeneous coordinates on $\PP^7$.
In Section \ref{prelim}, we collect
some preliminary results on the geometry of the quadric
$Q_6$ and of threefolds $Z\subset Q_6$ of bidegree 
$(1,p)$. In Section \ref{secmain}, we complete the proofs of our
main results, Theorems \ref{main} and \ref{mainsmooth}. 
Our arguments are completely classical in their nature.

{\bf Acknowledgements.} We would like to thank Igor Dolgachev for providing
useful references \cite{Gross} and \cite{Ran}. 
Lev Borisov also thanks Vasilii Iskovskikh for being an excellent 
undergraduate advisor at Moscow State University in 1987-1992.


\section{Statement of the results}\label{statements}
Throughout the paper all varieties are considered over
the field of complex numbers, and dimension means complex 
dimension.
The main object of our interest is a nonsingular quadric $Q_6$
of dimension $6$ and rank $8$ in the projective space $\PP^7$.
We also think of it as a nondegenerate bilinear form on the corresponding
vector space $\CC^8$, which allows us to define annihilators of
projective subspaces of $\PP^7$.

We recall some standard facts about the integer homology 
of $Q_6$. 
\begin{proposition}
The maximum dimension of a linear subspace $\PP^k$ in $Q_6$
is $k=3$. There are two types of such subspaces, which we
call \emph{horizontal} and \emph{vertical}. The spaces of the 
same type intersect each other in spaces of odd dimension
(empty, $\PP^1$ or $\PP^3$) and the spaces of different type
intersect each other in spaces of even dimension (point or $\PP^2$).
The homology group $H_6(Q_6,\ZZ)$ is freely generated by the 
classes of horizontal and vertical subspaces. 
\end{proposition}

\begin{proof}
See \cite[Chapter 6]{GH} for the first part of the statement. 
To show that the vertical and horizontal subspaces generate
$H_6(Q_6,\ZZ)$, it suffices to show that the latter is a
free abelian group of rank $2$ in view of the intersection pairing.
One can show this, for example, by realizing $Q_6$
as an orthogonal Grassmannian and decomposing it 
into a union of Schubert cycles.
\end{proof}

We will denote by $(a,b)$ the homology class given by $a$ horizontal
and $b$ vertical $\PP^3$-s. The main goal of this paper is to classify subvarieties 
$Z\subset Q_6$ of dimension three and homology class $(1,p)$.

\begin{remark}\label{explicit}
For the sake of definitiveness, we can assume that $Q_6$ is given
by the equation
\begin{equation}\label{q6}
x_1x_8-x_2x_7 + x_3x_6-x_4x_5=0
\end{equation}
in $\PP^7$ with homogeneous coordinates 
$(x_1\colon\ldots\colon x_8)$.
In this notation, particular examples 
of horizontal and vertical subspaces are given by 
$\{x_4=x_6=x_7=x_8=0 \}$ and $\{x_5=x_6=x_7=x_8=0 \}$, respectively.
Indeed, these subspaces intersect in a $\PP^2$, so
they are of different types, and it is up to us
to specify their names.
\end{remark}

We will now explain how one can generate examples 
of $Z\subset Q_6$ of bidegree $(1,p)$. We
refer the reader to \cite[Chapter II]{Hartshorne} for 
background on Cartier and Weil divisors. 

\begin{proposition}
Consider the dimension four singular quadric $Q_4$ 
of rank four given
by the equation
$x_3x_6=x_4x_5$ in the projective space $\PP^5$ with
homogeneous coordinates $(x_1\colon\ldots\colon x_6)$.
The class group of its Weil divisors is isomorphic 
to $\ZZ \oplus \ZZ$ and is generated by the 
divisors $D_1$ and $D_2$ which are cut out by the homogeneous
ideals  $\langle x_4, x_6\rangle$ and  $\langle x_5, x_6\rangle$ respectively.
\end{proposition}

\begin{proof}
The complement of $D_1\cup D_2$ in $Q_4$ is isomorphic to $\CC^4$ with
coordinates $(\frac {x_1}{x_6}, \frac {x_2}{x_6}, \frac {x_4}{x_6}, \frac {x_5}{x_6})$.
Consequently, any divisor on $Q_4$ can be pushed away from
this $\CC^4$ and written as a linear combination of $D_1$ and $D_2$,
see also \cite[Exercise II.6.5]{Hartshorne}.
\end{proof}

The quadric $Q_4$ is the intersection of the nonsingular quadric $Q_6$ given by 
\eqref{q6} and the dimension five subspace $x_7=x_8=0$. The rank of
the quadratic form drops to $4$, so this $Q_4$ is a double cone over a smooth 
$2$-quadric $Q_2$. 
\begin{proposition}
The images of the classes of $D_1$ and $D_2$ in $Q_6$ under 
the natural map $H_6(Q_4,\ZZ)\to H_6(Q_6,\ZZ)$ are 
given by $(1,0)$ and $(0,1)$, respectively. 
\end{proposition}

\begin{proof}
This follows from Remark \ref{explicit}.
\end{proof}

The following proposition describes the divisors $D$ of degree $(1,p)$ in $Q_4$
for $p\geq 1$. Observe that the divisor $D+(p-1)D_1$ is Cartier and is a multiple 
of the ample divisor $(1,1)$ on $Q_4$.
\begin{proposition}\label{classifyinQ4}
Every Weil divisor $D$ of degree $(1,p)$ on $Q_4$ can be written as the 
divisor of a polynomial $f(x_1,\ldots, x_6)$ minus $(p-1)D_1$,  with the polynomial
$f$ of the form
\begin{equation}\label{weil}
f= g_p(x_4,x_6)+\sum_{i=1,2,3,5} x_ig^{(i)}_{p-1}(x_4,x_6),
\end{equation}
where $g_p$ is homogeneous of  degree $p$ and $g_{p-1}^{(i)}$ are homogeneous
of degree $p-1$.
\end{proposition}

\begin{proof}
For each $p\geq 0$, we have the short exact sequence of sheaves
on $\PP^5$
$$
0\to {\mathcal O}(p-2) \to {\mathcal O}(p) \to {\mathcal O}_{Q_4}(p,p)\to 0.
$$
Since $H^1(\PP^5,{\mathcal O}(p-2))=0$ (see \cite[page 156]{GH}), 
the restriction map
\begin{align*}
H^0(\PP^5,{\mathcal O}(p)) \to H^0(Q_4,{\mathcal O}_{Q_4}(p,p)) 
\end{align*}
is surjective. Hence all Cartier divisors of type $(p,p)$ on $Q_4$
are given as zeroes of a degree $p$ polynomial $f(x_1,\ldots,x_6)$,
intersected scheme theoretically with $Q_4$.

It remains to investigate the condition on $f$ which says that
the corresponding divisor on $Q_4$ is of the form $D+(p-1)D_1$, that
is, that $f$ vanishes to order at least $(p-1)$ on $D_1$, modulo the
equation of $Q_4$.
We can use the equation of $Q_4$ to assume that no 
monomials of $f$ are multiples of $x_4x_5$, so that
$$f=g(x_1,x_2,x_3,x_5,x_6) + x_4 h(x_1,x_2,x_3,x_4,x_6),
$$
for some homogeneous polynomials $g$ and $h$ of degrees $p$ and $(p-1)$,
respectively. 
Consider the open subset $U$ of $Q_4$ given by
$x_5\neq 0$. This subset
is isomorphic to $\CC^4$ with coordinates
$$
(y_1,\ldots,y_4) = \left( \frac {x_1}{x_5},\frac {x_2}{x_5},
\frac {x_3}{x_5}, \frac {x_6}{x_5} \right),
$$
and the intersection of $D_1$ with $U$ is given by $y_4=0$.
Every polynomial $f$ gives the polynomial on $U$
\begin{align*}
f(x_1,\ldots, x_6) x_5^{-p} &= f(y_1,y_2,y_3, y_3 y_4, 1, y_4)\\
&=g(y_1,y_2,y_3,1,y_4) + y_3y_4 h(y_1,y_2,y_3,y_3y_4,y_4).
\end{align*}
In order for this polynomial to vanish to order $(p-1)$
at $y_4=0$, it should be divisible by $y_4^{p-1}$. The total degrees of the monomials in
$g$ are at most $p$, and the total degree of the
monomials in $y_3y_4 h$ are larger than $p$. Consequently,
there will be no cancellations among these monomials, and
each of them must be divisible by $y_4^{p-1}$. This means
that $g = (\mbox{linear}) \cdot x_6^{p-1}$, and the total degree of all monomials in $h$
in $x_4$ and $x_6$ is at least $p-2$. This means that the
total degree of all monomials in $f$ in $x_4$ and $x_6$
is at least $p-1$, which leads to \eqref{weil}.
\end{proof}

\begin{remark} 
\label{dcone}
If the variables $x_1$ and $x_2$ do not appear in 
(\ref{weil}), then the divisor $D$ will be a double cone over
a $(1,p)$ curve $C \subset Q_2$. Divisors of this form play a
crucial role in the study of orthogonal complex structures globally 
defined on $\mathbb{R}^6$, see \cite{BSV}. 
\end{remark}

When $p=3$, there is an irreducible subvariety of $Q_6$ of bidegree $(1,3)$
which does not lie in a $\PP^5\subset \PP^7$, and hence can not be of the type 
considered in Proposition  \ref{classifyinQ4}.
Geometrically, it can be described as follows. The intersection of the hyperplane
$x_8=0$ with $Q_6$ given by \eqref{q6}
is a cone over the Grassmannian $G(2,4)$, which is identified 
with a nonsingular dimension four quadric in $\PP^5$.
Let us recall the Pl\"ucker equation for
$G(2,4)$. If we have $w=\sum_{i<j}p_{ij}v_i\wedge v_j$, then the condition
 $w\wedge w=0$ means 
\begin{align*}
p_{12}p_{34} - p_{13}p_{24}+p_{23}p_{14}=0.
\end{align*}
We will henceforth view $G(2,4)$ as a subset of $Q_6$ by the
embedding
\begin{align*}
(p_{12},p_{13},p_{23},p_{14},p_{24},p_{34}) 
\mapsto (0, p_{12},p_{13},p_{23},p_{14},p_{24},p_{34}, 0).
\end{align*}
There is a surface $S$ of type $(1,3)$ in $G(2,4)$ namely  
the quadratic Veronese embedding $\PP^2 \hookrightarrow G(2,4)$
given by 
\begin{align}
\label{Veronese}
(u_0 \colon u_1 \colon u_2) \mapsto (u_0^2 \colon u_0 u_1 \colon u_0 u_2
\colon u_1^2 - u_0 u_2 \colon u_1 u_2 \colon u_2^2). 
\end{align}
The irreducible threefold of bidegree $(1,3)$ is then constructed as a 
cone over this surface, which is the image 
of the weighted projective space with variables $(u_0,u_1,u_2,u_3)$ with 
weights $(1,1,1,2)$ under the map given by 
\begin{align}
\label{coneV}
( u_0 \colon u_1 \colon u_2 \colon u_3 ) \mapsto 
(u_3 \colon u_0^2 \colon u_0 u_1 \colon 
u_0 u_2 \colon u_1^2 - u_0 u_2 \colon u_1 u_2 \colon u_2^2 \colon 0 ). 
\end{align} 
Indeed, by taking a cone over a surface of bidegree 
$(a,b)$ in $G(2,4)$ one gets a threefold of bidegree $(a,b)$ or
$(b,a)$ in $Q_6$, so the 
variety \eqref{coneV} is of bidegree $(1,3)$ or $(3,1)$. Consider 
its intersection
with the vertical $\PP^3$ given by 
\begin{equation}\label{test}
0=x_1=x_3=x_5=x_7,
\end{equation}
which occurs at the single point
$(u_0\colon u_1 \colon u_2 \colon u_3)=(1 \colon 0 \colon 0\colon 0)$.
It is easy to verify that the intersection is transversal, 
therefore it must be of bidegree $(1,3)$.

We are now ready to state the main theorem.
\begin{theorem}\label{main}
Every irreducible $Z$ of bidegree $(1,p)$ in $Q_6$ is given by one of the following,
up to the action of  ${\rm Aut}(Q_6)={\rm PSO}(8,\CC)$.
\begin{itemize}
\item
$p=0$ and $Z$ is a horizontal $\PP^3$.
\item
$p=1$ and $Z$ is a smooth quadric in $\PP^4\subset \PP^7$. 
\item
$p=3$ and $Z$ is the cone over the Veronese surface given by \eqref{coneV}.
\item
$p\geq 1$ and $Z$ is the Weil divisor  described in Proposition \ref{classifyinQ4}
in the rank four dimension four quadric $Q_4\subset Q_6$.
\end{itemize}
\end{theorem}
We will prove this theorem in Section \ref{secmain}. Meanwhile, 
we will make some remarks on this classification. 
\begin{remark}
As a corollary of our classification, we see that any threefold of order 
one in $Q_6$ is contained in a $\PP^5$, except in the third case
(the span of the $Z$ in this case is a $\PP^6$).
Furthermore, in view of Proposition \ref{classifyinQ4}, all
but the second and third cases are characterized by the condition 
that span of $Z$ is contained in a 
$\PP^5\subset \PP^7$ such that the annihilator of this 
$\PP^5$ with respect to $Q_6$ is an isotropic $\PP^1$. 
This is because if $Z$ is a smooth quadric in a $\PP^4$, then this 
$\PP^4$ can not be embedded into a $\PP^5$ with isotropic annihilator. 
Indeed, the annihilator
of a $\PP^4$ to which $Q$ restricts as a maximum rank quadric is a $\PP^2$ with 
a maximum rank quadric, which does not contain any lines.
\end{remark}

\begin{remark}
The automorphism group of $Q_4$ moves around the Weil divisors of type $(1,p)$ on it, and 
we have made no attempt to further restrict the shape of the equation $f$ in Proposition 
\ref{classifyinQ4} to classify the orbits of this action. 
From the point of view of orthogonal complex structures,
we are interested in classification of subvarieties of $Q_6$
up to the action of the real Lie group ${\rm{SO}}_{\circ}(7,1)$, which
is the subgroup of ${\rm{PSO}}(8,\CC)$ that preserves the twistor
fibration $ \PP^3 \to Q_6\to S_6$, see \cite{BSV}.
\end{remark}

Using this classification, we can characterize exactly which threefolds of 
order one are smooth. The proof of the following theorem 
is found in Section \ref{secmain}. 

\begin{theorem}
\label{mainsmooth}
There are only three smooth threefolds of order one in $Q_6$
up to ${\rm Aut}(Q_6)$, given in the following list.
\begin{itemize}
\item $p=0$ and $Z$ is a horizontal $\PP^3$.
\item $p=1$ and $Z$ is a smooth quadric in $\PP^4 \subset \PP^7$.
\item
$p=2$ and $Z$ is the Segre embedding 
$\PP^1\times \PP^2 \hookrightarrow \PP^5\subset \PP^7$ given by 
\begin{align*}
\big( (u_0 \colon u_1),(v_0\colon v_1\colon v_2) \big)
\mapsto (u_0v_0\colon u_1v_0\colon u_0v_1\colon
u_0v_2\colon u_1v_1\colon u_1v_2\colon 0\colon 0),
\end{align*}
which corresponds to $f=x_1x_6-x_2x_4$ in \eqref{weil}.
\end{itemize}
\end{theorem}


\section{Preliminary results on the geometry of $Q_6$ and $Z$}\label{prelim}
Here we collect various facts that will be useful later. 
For the remainder of the paper we will assume that $Z$ is an
irreducible threefold of bidegree $(1,p)$ in $Q_6$. 

\begin{proposition}\label{degree}
If $Z$ has homology class $(1,p)$, then it has degree $p+1$ in $\PP^7$.
\end{proposition}

\begin{proof}
Horizontal and vertical $\PP^3$-s have bidegree $(1,0)$ and $(0,1)$
respectively and degree one. It remains to observe that degree
is a linear function on $H_6(Q_6,\ZZ)$.
\end{proof}

\begin{proposition}\label{basics}
The space of vertical (or horizontal) $\PP^3$-s in $Q_6$ is of dimension six,
and is isomorphic to a nonsingular quadric.
The space of vertical (or horizontal) $\PP^3$-s through a given point in $Q_6$
is of dimension $3$ and is isomorphic to a $\PP^3$.
The space of vertical (or horizontal) $\PP^3$-s through a given $\PP^1\subset Q_6$ is of dimension one and is isomorphic to a $\PP^1$. 
There is exactly one vertical and one horizontal $\PP^3$ through a given $\PP^2\subset Q_6$.
\end{proposition}

\begin{proof}
This follows from the classical description of the 
isotropic Grassmannians of low dimension. The first space is 
described by oriented or reverse oriented isotropic $4$-planes in $\CC^8$, 
which are determined by a projective pure spinor. In dimension 
$8$, a pure spinor satisfies a single quadratic relation of rank $8$, so these 
spaces are each nonsingular quadrics of dimension~$6$, see \cite[Chapter VI]{Cartan}. 
For the last three spaces of the proposition, 
we need to consider oriented isotropic $4$-planes 
in $\CC^8$ containing a fixed line $P_1$, plane $P_2$, or $3$-plane $P_3$,
respectively. This is equivalent to 
choosing an oriented isotropic $(4-j)$--plane $V_j \subset \mbox{Ann}(P_j)/ P_j = \CC^{2(4-j)}$. 
This space is a $\PP^3$, $\PP^1$, or a single point since all non-zero spinors 
are pure in dimensions $6$, $4$, and $2$, respectively, see \cite[Chapter VI]{Cartan}.  
\end{proof}

\begin{lemma}\label{sing}
If $p\neq 1$ and $Z$ is contained in a $\PP\cong \PP^6\subset \PP^7$, then $Q_6\cap \PP$ is singular.
\end{lemma}

\begin{proof}
Suppose $Q_6\cap \PP$ is a smooth $5$-quadric $Q_5$. By the 
Lefschetz hyperplane theorem and Poincar\`e duality, 
$H_6(Q_5, \ZZ)  \cong \ZZ$, with the generator given by 
the intersection of $2$ hyperplanes with $Q_5$. Consider the
image of $H_6(Q_5, \ZZ)$ in $H_6(Q_6,\ZZ)$ under the 
inclusion map. The generator maps to the intersection of 
$3$ hyperplanes with $Q_6$, which is $Q_6 \cap \PP^4$. 
We may choose the $\PP^4$ so that the rank of the quadratic form  
drops to $2$, thus the generator will map to a $3$-quadric 
of rank $2$ which is the union of a horizontal $\PP^3$ 
and a vertical $\PP^3$ and is therefore of bidegree $(1,1)$.
Consequently, the image of $H_6(Q_5, \ZZ)$ in $H_6(Q_6,\ZZ)$ 
consists of classes of type $(m,m)$.
Since the map from $Z$ to $Q_6$ factors through 
$Q_5$, the class of $Z$ is of type $(m,m)$, so $p=1$, contradiction.
\end{proof}

The following trick is one of the key technical tools of the paper.
\begin{lemma}\label{keytrick}
Let $Z$ be a dimension three subvariety of $Q_6$ of bidegree $(1,p)$.
Any vertical $\PP^3$ either intersects $Z$ transversely in a 
single point, or contains
at least a one-dimensional set of points of $Z$.
\end{lemma}

\begin{proof}
Suppose that 
$Z$ intersects the vertical $\PP^3$ at a disjoint set of points $\{z_1,\ldots,z_k\}$.
Then the intersection of the corresponding homology classes equals to the sum
of  local contributions. On one hand, this intersection equals $1$. On the other hand,
every transveral intersection point contributes $1$ and every non-transversal
isolated intersection point contributes strictly more than $1$, see 
\cite[Proposition 8.2 (a) and (c)]{Fulton}.
\end{proof}

A typical application of Lemma \ref{keytrick} is the following.
\begin{lemma}\label{singular}
If $A$ is a singular point of $Z$, then $Z\subset \Ann(A)$. 
\end{lemma}

\begin{proof}
Consider the vertical $\PP^3$-s that contain $A$. Each of them intersects
$Z$ nontransversely at $A$, consequently, each of them must contain
at least  a dimension one subset of $Z$. There is a dimension three space 
of such $\PP^3$-s, so the space of pairs 
\begin{align*}
W  \equiv \{ (P, x) \ | \
P \mbox{ is a vertical } \PP^3 \mbox{ containing } A, \mbox{ and } x \in P \cap Z \},
\end{align*}
is at least four-dimensional. Consider the projection of this space
to the second factor $\pi_2 \colon W \rightarrow Z$, 
and observe that the fibers outside of 
$\Ann(A)\cap Z$ are empty.
Fibers over points of $\Ann(A)\cap Z$ are one-dimensional by Proposition
\ref{basics}. 
This shows 
that $\Ann(A)\cap Z$ must be three-dimensional, so $Z\subset \Ann(A)$ since
$Z$ is irreducible.
\end{proof}
The following proposition gives another restriction 
on the intersection of $Z$ with a vertical $\PP^3$.
\begin{proposition}
If $Z$ intersects a vertical $\PP^3$ in a set of dimension two, 
then this set must consist of a $\PP^2$, and perhaps some
lower dimensional irreducible components. 
\end{proposition}
\begin{proof}
Denote this vertical $\PP^3$ by $P$, and let $Z_0$ be
the union of the $2$-dimensional irreducible components of $P \cap Z$. 
Assume that the total degree of $Z_0$ is $ d > 1$. 
Consider the space $S$ of all vertical $\PP^3$-s in $Q_6$ which intersect 
$P$ in a line. By Proposition \ref{basics}, this is a 5-dimensional manifold. 
Define
\begin{align*}
W \equiv \{(P_1, z ) \ | \ P_1 \in S
\mbox{ and } z \in P_1\cap (Z \setminus Z_0) \}.
\end{align*}
We claim that for generic 
$P_1 \in S$, $\dim\big(P_1\cap (Z \setminus Z_0) \big) \geq 1$. 
Indeed, otherwise $P_1$ and $Z$ intersect at isolated points,
provided we choose $P_1$ so that the corresponding 
line in $P$ intersects $Z_0$ in isolated points. We can account 
for $d$ points of intersection
on $Z_0$, and all other points of intersection contribute positively, 
which contradicts the fact that $Z$ is of order one.
This implies that $\dim(W) \geq 6$. 

On the other hand, $W$ projects to $Z \setminus Z_0$, which is of dimension three.
The fibers of this projection are of dimension at most two, since any $P_1 \in S$ 
containing $z$ is completely determined by choice 
of a line in the two-dimensional subspace $\Ann(z)\cap P$.
This implies that $\dim(W) \leq 5$, which is a contradiction. 
Consequently, if $Z_0$ is nonempty then it has degree one, 
which implies that it is a linear subspace \cite[page 174]{GH}. 
\end{proof}
The following two lemmas describe the intersection of $Z$ with a
horizontal $\PP^3$. 
\begin{lemma}\label{withhp}
If $Z$ intersects a horizontal $\PP^3$ (set-theoretically) at a set of 
dimension one, then this set must consist of one line in $\PP^3$ and 
perhaps some disjoint points.
\end{lemma}
\begin{proof}
Denote this horizontal $\PP^3$ by $P$, and
denote by $C$ the union of the dimension one components of $Z\cap P$.
Define 
\begin{align*}
W' \equiv \{ (P', z) \ | \ P' \mbox{ is a vertical } \PP^3 \mbox{ with } 
\dim(P' \cap P) = 2, \mbox{ and } z \in (Z \cap P') \setminus P\}. 
\end{align*}
If the degree of $C$ (in $P = \PP^3$) is strictly larger than one, 
it means that a generic $P'$ as above intersects $C$ in at least two points,
since the degree of a curve in $P$ is given by counting transverse 
intersections with a $\PP^2 \subset P$. 
By Lemma \ref{keytrick}, a generic such $P'$ contains at least a curve 
of points of $Z$.
This means that the generic fiber of the projection on the 
first factor is at least $1$ dimensional. 
By Proposition \ref{basics}, there is a dimension three space
of such $P'$, so the dimension of $W'$ is at least $4$. 
However, all fibers under the second projection 
are points. This follows because for $z \in Z \setminus P$, 
$\dim(\Ann(z)\cap P) = 2$,
so the fiber over $z$ is the vertical 
$\PP^3$ given by $\Span(z,\Ann(z)\cap P)$.
Since $\dim Z=3$, this implies that $\dim (W') \leq 3$, 
a contradiction. 
\end{proof}
\begin{lemma}
\label{horiz}For $p > 0$, the generic horizontal $\PP^3$ intersects 
$Z$ transversely at $p$ disjoint smooth points.
\end{lemma}
\begin{proof}
Consider the space of pairs $(z, P)$ where $z\in Z$ and $P$ is a horizontal $\PP^3$ that contains
$z$. It is a $\PP^3$-bundle over $Z$, so it is of dimension $6$. Consider the projection
to the space of all horizontal $\PP^3$. Since $p>0$, the fibers of this projection 
are nonempty. Since it is a map between projective varieties of the same 
dimension (six), it is generically finite.
The generic fiber will consist of smooth points, which means that the generic horizontal $\PP^3$, call it $H$, intersects $Z$ transversely at smooth points. 
By the definition of bidegree, $H$ will then necessarily intersect $Z$ 
at $p$ disjoint points.
\end{proof}
The following construction will also be used in the proof of our main theorem. 
\begin{lemma}
\label{Bl}
Let $H$ be a horizontal $\PP^3$, and consider the blowup ${\rm{Bl}}(Q_6, H)$ 
of $Q_6$ along $H$. 
Then ${\rm{Bl}}(Q_6, H)$ is naturally identified with a bundle
$\PP^3 \rightarrow B \overset{\pi}{\rightarrow} \PP^3 = H^*$, which is defined 
by specifying that a fiber of this bundle is the vertical
$\PP^3$ that intersects $H$ at a given $\PP^2 \in H^*$ (see Proposition \ref{basics}). 
A point $z \notin H$ is mapped to the $2$-plane $Ann(z) \cap H$. 
This map extends smoothly to the exceptional 
divisor, which is the projectivized normal bundle of $H$.  
The image of the exceptional divisor over a point $z \in H$ 
consists of all vertical $2$-planes which intersect 
$H$ in a $\PP^2$ containing $z$. 
\end{lemma}
\begin{proof}
It is well-known that the blowup ${\rm{Bl}}(\PP^7,H)$ of $\PP^7$ at $H$ maps to 
$\PP^3$ that parametrizes the $\PP^4$-s in $\PP^7$ that contain $H$. 
The blowup of $H$
in $Q_6$ is the proper preimage of $Q_6$ in ${\rm{Bl}}(\PP^7,H)$, so 
it projects naturally to $\PP^3$. Note that for each $\PP^4$ that 
contains $H$, its intersection with $Q_6$ is the union of $H$ and a 
vertical $\PP^3$ that intersects $H$ at a $\PP^2$. The lemma follows 
easily from these facts.  
\end{proof}
\section{Proof of the main theorems}
\label{secmain}

\begin{theorem}\label{smallpinP5}
Let $Z$ be a subvariety of $Q_6$ of bidegree $(1,p)$. Let us further assume 
that $Z$ is singular 
or that $Z$ is smooth and $p\leq 3$. Then for $p\neq 3$, $Z$ is contained  
in a $\PP^5$. For $p=3$, either $Z$ is contained in a $\PP^5$, or $Z$
is equivalent to the Veronese embedding \eqref{coneV} under an automorphism 
of $Q_6$. 
\end{theorem}

\begin{proof}
From Proposition \ref{degree}, $\deg(Z)=p+1$,
consequently, $Z$ is contained in a $\PP^{p+3}$,  see \cite[page 174]{GH}.
This implies the statement for $p<3$. 
For $p=3$ similarly $Z$ is contained in a $\PP^6$.
If $Z$ is singular then $Z$ is contained in a $\PP^6$ by Lemma \ref{singular}.
In either case, by Lemma \ref{sing}, there is a point $A$ in $Q_6$ such 
that this $\PP^6$ is $\Ann(A)$ and 
$Q_6\cap \PP^6$ is a cone with vertex $A$ over the smooth quadric in $\PP^5$, 
which we will identify with the 
Grassmannian $G(2,4)$. Consider the projection from $Z \setminus A$ 
to $G(2,4)$ from the point $A$. 

First, consider the case when the image is a threefold $Y$ (if $A$ is 
contained in $Z$ then we consider the closure of the image of $Z \setminus A$). 
The divisor $Y$ is a zero set of a section of a line bundle ${\mathcal L}=\pi^*{\mathcal
O}_{\PP^5}(d)$ for some $d>0$, where $\pi\colon G(2,4)\to\PP^5$ is the Pl\"ucker
embedding. 
If $d > 1$, for any smooth point $x\in Y$ consider the space
\begin{align*}
W_x \equiv \{ H = \PP^2 \subset G(2,4) \mbox{ with } x \in H 
\mbox{ and } {\rm Span}(H,A) \mbox{ is a horizontal } \PP^3 \}.
\end{align*}
The span of the tangent spaces at $x$ 
of all such $H \in W_x$ 
is the entire $T(x,G(2,4))$. As a result, a generic $H$ will have a point
$x\in H\cap Y$ such that $T(x,H)$ that does not lie in $T(x,Y)$, and thus 
$H$ and $Y$ will intersect transversely near $x$ inside $G(2,4)$. Consequently, 
the restriction to $H$ of the
section of $\mathcal L$ whose zero set is $Y$, is smooth at $x\in H$. On the other
hand, the pullback of $\mathcal L$ to $H$ is ${\mathcal O}_H(d)$, which shows that the
set-theoretic intersection of $Y$ and $H$ is a (perhaps reducible) curve $C$ 
of total degree strictly larger than one. From transversality at $x$, 
we know that $Y$ does not intersect $H$ at a multiple of a 
single line, thus $C$ is not a line. Since $H$ is generic, 
we may assume that $P={\rm Span}(H,A)$ intersects
$Z$ at a curve $\tilde{C}$, rather than a surface (and perhaps some isolated points). 
Since the curve $\tilde{C}$ covers $C$ under projection, 
it cannot be a line either. This contradicts Lemma~\ref{withhp},  therefore $d = 1$
and $Y$ is contained in a $\PP^4$, so $Z$ is contained in a $\PP^5$ as claimed.
 
Otherwise, the image of $Z$ must be a surface $S$, and $Z$ must be a cone 
over that surface with vertex $A$. 
Since $Z$ is irreducible, this surface must be an irreducible surface 
of bidegree $(1,p)$ in $G(2,4)$.
From \cite{Ran}, $S$ is either of the form $S(M,F)$ which is 
by definition contained in the hyperplane in $\PP^5$ determined 
by the line $M$, or $S$ is equivalent under an automorphism of $G(2,4)$
to the $(1,3)$ surface given by the secant line image of the twisted cubic. 
Under the Pl\"ucker embedding $G(2,4) \subset \PP^5$, 
the image of this surface is the Veronese surface (\ref{Veronese}). 
To see this, the (affine part of the) rational normal cubic is the set of points 
$\{ (1\colon t\colon t^2\colon t^3) \ | \ t\in \CC \} \subset \PP^3$, so we can parametrize 
a dense open subset of the  surface $S$
by 
\begin{align*}
\{ (1\colon t\colon t^2\colon t^3)\wedge (1\colon s\colon s^2\colon s^3)
\colon (s,t)\in\CC^2 \},
\end{align*}
which gives 
\begin{align*}
(x_2\colon \ldots \colon x_7) &=
(p_{12}\colon p_{13}\colon p_{23}\colon p_{14}\colon p_{24}\colon p_{34})\\
&=
(s-t\colon s^2-t^2\colon 
ts^2-st^2\colon 
s^3-t^3 \colon 
ts^3-st^3\colon t^2s^3-s^2t^3)\\
&= (1\colon s+t\colon 
ts\colon s^2+st+t^2\colon 
ts(s+t)\colon t^2s^2).
\end{align*}
This can then be parametrized in terms of the symmetric functions 
$(u_1,u_2)=(s+t,st)$ as
\begin{align*}
 (1\colon u_1\colon u_2 \colon u_1^2-u_2 \colon u_1u_2\colon u_2^2).
\end{align*}
Finally, we homogenize to obtain (\ref{Veronese}),
and observe that any automorphism of $G(2,4)$ extends to an 
automorphism of $Q_6$ which maps $A$ to the point 
$(1 \colon 0 \colon \ldots \colon 0)$.   
\end{proof}
The following key result allows us to deal with smooth
$Z$ for $p>3$.
\begin{theorem}\label{p4inP5}
If $Z$ is a smooth subvariety in $Q_6$ of bidegree $(1,p)$ with $p>3$ 
then $Z$ is contained in a $\PP^5$.
\end{theorem}
\begin{proof}
We take a generic $H$ as in Lemma \ref{horiz}, which intersects $Z$ at 
$p$ disjoint points, call these points $\{z_1,\ldots,z_p \}$.
Consider ${\rm{Bl}}(Q_6, H)$, and the corresponding proper preimage $Y$ of $Z$.
The blowup of $Q_6$ at $H$ induces a map $Y\to Z$ which is the blowup of 
$Z$ at the scheme-theoretic
intersection of $H$ and $Z$. Since $H$ intersects $Z$ transversely, 
$Y\to Z$ is the blowup 
of $Z$ at the points $z_i$.  The exceptional divisors 
$E_i$ of $Y\to Z$ are isomorphic to $\PP^2$. Consider the restricted 
projection map $\pi\colon Y\to \PP^3$. 
As noted in Lemma \ref{Bl}, it maps each $E_i$ isomorphically
to a $\PP^2$ in $\PP^3$ that could be identified with the space of 
vertical $\PP^3$-s that intersect $H$ in a $\PP^2$ containing $z_i$.

The fibers of $\pi$ are nonempty, so the generic fiber is smooth of dimension zero, 
so it is a point by Lemma \ref{keytrick}.
Consequently, $\pi\colon Y\to \PP^3$ is a birational map of smooth varieties, 
which implies that its fibers are connected
by Zariski's Main Theorem \cite[Corollary III.11.4]{Hartshorne}. 
Consider the exceptional divisors of $\pi$ on $Y$. 
We claim that for each pair $(i,j), i\neq j$
there is an exceptional divisor $D_{ij}$ of $\pi$ which is contracted to the 
line $l_{ij}=\pi(E_i)\cap \pi(E_j)$.
Indeed, for every point in $l_{ij}$ the corresponding vertical $\PP^3$ that 
intersects $H$ at a $\PP^2$
contains both $z_i$ and $z_j$. By Lemma \ref{keytrick}, the fiber over this 
point is of dimension at least one.
Because the fibers of $\pi$ are connected, for each $j\neq i$ there is an exceptional
divisor $D_{ij}$ of $\pi$ which maps to $l_{ij}$ and intersects $E_i$ at a curve.

We next consider $K_Y$, the canonical class of $Y$.
Since $\pi$ is a birational map, we have the equivalence 
\begin{align}
\label{KM}
K_Y = \pi^*K_{\PP^3} + \sum_r a_r D_r,
\end{align}
for the exceptional divisors $D_r$ of $\pi$,
where the discrepancies $a_r$ are the orders of vanishing
of the Jacobian, and hence are positive integers.
Since $E_i$ is obtained by blowing up a point on a 
smooth threefold, we have $K_Y\vert E_i = -2l$ where $l$ is the class of a line on $E_i$.
Restricting (\ref{KM}) to $E_i$, we then have
\begin{align*}
-2l = K_Y\vert_{E_i} =-4l+ \sum_r a_r D_r\vert_{E_i}.
\end{align*}
This shows that the total degree of intersection of $D_r$ with $E_i$ is
at most two. However, for all $j\neq i$, the divisors $D_{ij}$ intersect $E_i$, and
the number of these divisors is at least $p-1>2$. This means that for each $i$ the number
of distinct divisors $D_{ij}$ is at most two. As a consequence, the number of distinct lines
$l_{ij}$ is at most two. This can be rephrased as saying that for each $i$ the number of distinct
lines $\Span(z_i,z_j)$ in $H$ is at most two. 

We claim that this means that all $z_i$ lie in a line. Indeed, consider the line that 
contains the maximum number of $z_i$. If it contains only two $z_i$, then there are
strictly more than two lines through $z_1$, since $p>3$. Else, 
if there is a point $z_j$ 
not on this line, then there are at least three lines through that $z_j$. 
In either case, there is a point with at least $3$ distinct lines
passing through it, contradicting the conclusion from the preceeding 
paragraph. Denote the line containing all the $z_i$ by $L$.

Let us now consider the projection of $Z_1=Z - \cup_iz_i$ from the line $L$.
That is, we let $W_1 \cong \PP^5$ denote the space of $\PP^2$-s in $\PP^7$ 
containing $L$, and to each point $x$ in $Z_1$ we associate the $\PP^2 \in W_1$ 
which is the span of $x$ and $L$.
We claim that the 
dimension of the image of $Z_1$ is three. Indeed, the 
dimension of the image is at most three. If the dimension is strictly 
less than three, then a generic $\PP^2 \in W_1$, call it $Q$, 
does not intersect the image of $Z_1$. 
The subspace $Q$ is a $\PP^2$-family of
$\PP^2$-s, each of which contains $L$.
The image of $Z_1$ under the projection is all the $\PP^2$-s containing $L$ 
which hit a point of $Z_1$. This means that each
$\PP^2$ in our $\PP^2$-family of $\PP^2$-s does not hit any
point of $Z_1$. The union of all of these $\PP^2$-s
is a $\PP^4$. This $\PP^4$ therefore only hits $Z$ on $L$, 
that is, this $\PP^4$ hits $Z$ exactly at
the points $z_i$, with transverse intersection since $Q$ is generic.
This shows that the degree of $Z$ in $\PP^7$ is $p$. However, this degree is 
$p+1$, see Proposition \ref{degree}. This contradiction proves that 
the image of $Z_1$ is of dimension three.

Next, we claim that the degree of the closure of the 
image of $Z_1$ is one. We choose a $\PP^2 \in W_1$, 
call it $Q$, which hits $\Im(Z_1)$ transversely in $k$ points, 
where $k$ is the degree of $\overline{\Im(Z_1)} \subset \PP^5$. 
This means that the number of $\PP^2$-s in $Q$ 
which intersect $Z_1$ is exactly $k$. The corresponding 
$\PP^4 \subset \PP^7$ then intersects $Z$ in
$p + r k$ points, where $r$ is the degree of
the projection map.
We must therefore have $p + r k = p +1 $, which implies that 
$k = 1$. 
Consequently, $\overline{\Im(Z_1)}$ is a $\PP^3$ in $W_1$, 
which implies that $Z$ is contained in the corresponding 
$\PP^5$ in $\PP^7$.
\end{proof}
\begin{remark} In the proof of Theorem \ref{mainsmooth} below, 
we will see that the case considered above in 
Theorem \ref{p4inP5} does not actually occur. 
Furthermore, it would not be difficult to extend the above proof 
to cover the singular case. However, it is more convenient
to quote the result from \cite{Ran} to deal with this case, 
as we have done in the proof of Theorem \ref{smallpinP5} above. 
\end{remark}
We are now ready to prove our main results.
\begin{proof}[Proof of Theorem \ref{main}.]
By Theorems  \ref{smallpinP5} and \ref{p4inP5}, the variety  $Z$ is either
given by \eqref{coneV} or is contained inside a $\PP^5\subset \PP^7$. If $p=1$, then it 
is contained in a $\PP^4$ in $\PP^7$. This $\PP^4$ intersects $Q_6$ in a dimension 
three quadric of rank at least two. However, rank two is not 
possible since this intersection is the union of a 
horizontal and a vertical $\PP^3$, which is reducible. 
The rank could then be anywhere from $3$ to $5$. When the rank is 
three or four, the annihilator of $\Span(Z)$
contains an isotropic line, in which case $Z$ can be moved inside 
$Q_4$ by an automorphism of $Q_6$, since the automorphism 
group acts transitively on isotropic lines. If the rank is five, than $Z$ 
is of the second case listed in Theorem \ref{main}.

If $p\neq 1$ and $Z$ is contained inside a $P\cong \PP^5\subset \PP^7$, 
then $Q_6\cap P$ is a singular quadric
of rank $4$. To see this, the rank of the quadratic form 
restricted to $P$ can be either $4, 5,$ or $6$.  
If the rank is $5$ or $6$, then consider $\Ann(P)\cong \PP^1$. 
Rank $4$ is equivalent to $\Ann(P) \subset Q_6$, so 
choose a point $x \in \Ann(P) \setminus Q_6$.  
For this point $x$, $Z\subset \Ann(x)\cong \PP^6$, and $\Ann(x)\cap Q_6$ is smooth, 
which contradicts Lemma \ref{sing}. Noting that all singular quadrics of rank 
four and dimension four in $Q_6$ can be moved to $Q_4$ by an automorphism 
of $Q_6$, the theorem follows from Proposition \ref{classifyinQ4}.
\end{proof}
We next find exactly which threefolds of order one are
smooth. The proof amounts to directly verifying which 
divisors in Proposition \ref{classifyinQ4} are smooth.  
\begin{proof}[Proof of Theorem \ref{mainsmooth}.]
For $p=0$ the statement is clear. For $p=1$ the only possibilities
for $Z$ are the quadrics of ranks $3$, $4$ and $5$, of which only
the quadrics of rank $5$ are smooth. All rank $5$ quadrics in
$Q_6$ are equivalent under the action of ${\rm Aut}(Q_6)$.
The third case in Theorem \ref{main} is singular at the cone point. 
Therefore we need only consider the last case in Theorem \ref{main}. 

We first consider the case $p>2$. We may assume that
$Q_4$ is given by 
\begin{align*}
Q_4 = \{x \in \PP^7 \ | \ x_7=x_8=0,x_3x_6=x_4x_5 \}, 
\end{align*}
and by Proposition \ref{classifyinQ4}
that $Z$ is a Weil divisor on $Q_4$ such that $Z+(p-1)D_1$ is given by
\begin{align}
\label{deq}
0=g_p(x_4,x_6)+\sum_{i=1,2,3,5}x_ig_{p-1}^{(i)}(x_4,x_6).
\end{align}
For any $(a,b)\in \CC^2 \setminus (0,0)$  consider the $\PP^3 \subset Q_4$ given by
\begin{align*}
P_{(a,b)} = \{ (u_1 \colon u_2 \colon u_3 a \colon u_4 a \colon u_3 b \colon u_4 b)
\ | \
(u_1 \colon u_2 \colon u_3 \colon u_4) \in \PP^3\},
\end{align*}
which contains the singular line of $Q_4$
\begin{align*}
L=\{ (u_1 \colon u_2 \colon 0 \colon \ldots \colon0)
\ | \
(u_1 \colon u_2) \in \PP^1 \}. 
\end{align*}
Clearly, $P_{(a,b)}$ depends only upon the class $(a \colon b) \in \PP^1$. 
The Weil divisor
$Z$ of class $(1,p)$ is a Cartier divisor on $Q_4 \setminus L$, hence
it restricts to a Cartier divisor on 
$P^{*}_{(a,b)} \equiv P_{(a,b)} \setminus L$. The groups of
Cartier divisors on $P_{(a,b)}$ and $P^{*}_{(a,b)}$ are isomorphic, and
it is easy to see that the class of $Z$ when restricted to
$P^{*}_{(a,b)}$ is ${\mathcal O}(1)$.

Explicitly, the equation (\ref{deq}) when restricted to $P_{(a,b)}$ 
can be written
\begin{align}
0=u_4^{p-1} \left( g_{p-1}^{(1)}(a,b) u_1 + g_{p-1}^{(2)}(a,b) u_2 +
  c_3 u_3  + c_4 u_4 \right), 
\end{align}
where $c_3$ and $c_4$ are constants. 
The factor $u_4^{p-1}$ amounts to the restriction of $D_1$ to 
$P^{*}_{(a,b)}$, 
so $Z \cap P^*_{(a,b)}$ is given by the 
second, linear, factor. 
Since $Z$ is irreducible, it can never contain a $\PP^3$. Hence,  
\begin{align}
\label{qp}
Q_{(a,b)} = \overline{ Z \cap P^*_{(a,b)}}
\end{align}
is a $\PP^2$ contained in $Z$.
This $\PP^2$ contains $L$ if and only if
both $g_{p-1}^{(1)}(a,b)$ and $g_{p-1}^{(2)}(a,b)$ are zero. 
If they are not both zero, then $Q_{(a,b)}$ passes through a
unique point on $L$ given by 
\begin{align}
\label{pointL}
Q_{(a,b)} \cap L = (-g_{p-1}^{(2)}(a,b) \colon g_{p-1}^{(1)}(a,b) \colon 0
\colon 0 \colon 0 \colon 0). 
\end{align}

We are now ready for the main argument.
View $g_{p-1}^{(1)}$ and $g_{p-1}^{(2)}$ as sections 
of $\mathcal{O}(p-1)$ on $\PP^1$.
If $g_{p-1}^{(1)}$ and $g_{p-1}^{(2)}$  have no common zeroes, 
then we can define the map $\psi \colon \PP^1 \rightarrow L $ by
\begin{align}
\label{psi}
\psi \big(  (a \colon b) \big) = 
\big( -g_{p-1}^{(2)}(a,b) \colon g_{p-1}^{(1)}(a,b) \colon 0 \colon 0
\colon 0 \colon 0 \big).
\end{align} 
The map $\psi$ has degree $p - 1$, and since $p - 1 > 1$, 
generically every value has at least two preimages.
Consequently, for a generic point $z \in L$, there exist two distinct 
$\PP^3$-s as above, $P_1 = P_{(a_1,b_1)}$ and $P_2 = P_{(a_2,b_2)}$, 
such that $z\in Q_{(a_1,b_1)}$ and 
$z\in Q_{(a_2,b_2)}$.  Since $P_1\cap P_2=L$, 
the corresponding $\PP^2$-s in (\ref{qp}) intersect transversely. 
If $Z$ were smooth at $z$ then $T_z Z$, 
the tangent space to $Z$ at $z$, would contain both planes, so its dimension 
would be at least four. This contradiction shows that $Z$ is singular at $z$. 

We next consider the cases where the map $\psi$ is not well-defined. 
If $g_{p-1}^{(1)}$ and $g_{p-1}^{(2)}$ are both zero, 
then $Z$ is a double cone over a $(1,p)$ curve in a
smooth $2$-quadric (see Remark \ref{dcone}), 
which is singular for $p \geq 1$. 
There remains the case that $g_{p-1}^{(1)}$ and $g_{p-1}^{(2)}$ 
have a common root but are not both zero. Denote 
this common root by $(a_0, b_0)$.
Consider the $\PP^3$ corresponding to this root, $P_{(a_0,b_0)}$,
and another $\PP^3 = P_{(a,b)}$ for some generic $(a_1,b_1)$ which is different 
from this root and has one of $g_{p-1}^{(1)}$ and $g_{p-1}^{(2)}$ nonzero. 
Denote by $z$ the intersection point of $Q_{(a_1,b_1)}$ 
and $L$ given in (\ref{pointL}). 
Since $P_{(a_0,b_0)}$ contains $L$, it also passes 
through $z$. We have found two
$\PP^2$-s contained in $Z$ and passing through the 
point $z$. 
We again argue that they are transversal. Indeed, the intersection of $P_{(a_0,b_0)}$ 
and $P_{(a_1,b_1)}$
is $L$, so the intersection of the two $\PP^2$-s is contained in $L$. 
However, $L$ is not contained in the second $\PP^2$, hence the two $\PP^2$-s through
$z$ intersect at $z$ only. 
As above, this gives a contradiction.
This finishes the proof that $Z$ is singular for $p > 2$. 

In the $p=2$ case, we can argue as in the preceeding 
paragraph to assume that $g_1^{(1)}$
and $g_1^{(2)}$ are both nonzero 
and have no common zeroes.
By a linear change of $x_1$ and $x_2$ (which easily
extends to an automorphism of $Q_6$ by an appropriate linear
change of $x_7$ and $x_8$) we may assume that 
the equation (\ref{deq}) has the form
\begin{equation}\label{almostSegre}
0=x_1x_6-x_2x_4- x_4 h_2(x_3,x_4,x_5,x_6) 
+ x_6h_1(x_3,x_4,x_5,x_6),
\end{equation}
where $h_1$ and $h_2$ are linear. We claim that $h_i$ 
can be absorbed into $x_i$ by a linear change of coordinates
which keeps $Q_4$ and $Q_6$ unchanged. If $h_1 = a_3 x_3 + a_4 x_4
+ a_5 x_5 + a_6 x_6$,  then we can use 
\begin{equation}\label{sym}
(x_1,\ldots,x_8)\to (x_1 - a_3x_3 -a_5x_5,
x_2,x_3,x_4 - a_5x_8,x_5,x_6 + a_3x_8,x_7,x_8),
\end{equation}
and similarly for $h_2$ to reduce \eqref{almostSegre}
on $Q_4$ to
\begin{align*}
0=x_1x_6-x_2x_4 - x_4 \tilde h_2(x_4,x_6) 
+ x_6\tilde h_1(x_4,x_6),
\end{align*}
and then use transformations analogous to \eqref{sym} to eliminate 
all of the extra terms.
This case is finished by observing that the 
equation $x_1 x_6 - x_2 x_4 = 0$ in $Q_4$
cuts out $D_1$ and the image of the Segre embedding,
so $Z$ is as claimed.  
\end{proof}
The technique in the above proof yields a geometric
description of the last case in Theorem \ref{main}.
\begin{proposition} If $Z$ is in the last case in Theorem \ref{main}, 
then $Z$ contains $L$ and is equal to the union of the $Q_{(a,b)}$ 
over all $(a \colon b) \in \PP^1$. This is almost a disjoint union, 
in the sense that the only intersections occur at points of $L$.
Furthermore, $Z$ is smooth away from $L$.
\end{proposition}
\begin{proof}
Either at least one of the $Q_{(a,b)}$ contains $L$
or all the points of $L$ have a $Q_{(a,b)}$ through it,
so the union of $Q_{(a,b)}$ contains $L$. The complement
of $L$ in $Q_4$ is the (disjoint) union of $P_{(a,b)}^*$, and each
of them intersects $Z$ in an open subset of the corresponding
$Q_{(a,b)}$. This shows that $Z$ is equal to the union of $Q_{(a,b)}$,
and clearly the only intersections occur at points of $L$.
Finally, if $Z$ were singular at some point $A$ outside of $L$, then by 
Lemma \ref{singular}, $Z$ would be contained in 
$\Ann(A)\cap Q_4$, which is just a union of a vertical 
and a horizontal $\PP^3$.
\end{proof}
\begin{remark} The above 
proof of Theorem \ref{mainsmooth} also leads to a criterion for when the 
polynomial $f$ in \eqref{weil} leads to an irreducible divisor of bidegree $(1,p)$.
If the corresponding $Z$ is reducible, then the $(1,p)$ Weyl divisor has to split up
as a $(1,p_1)$ and some $(0,1)$ divisors, which are necessarily vertical
subspaces. Therefore the condition that all $Q_{(a,b)}$ are $\PP^2$-s is both 
necessary and sufficient.
\end{remark}

\bibliographystyle{amsalpha} 
\bibliography{Z1p_references}

\end{document}